\UseRawInputEncoding
\documentclass[12pt,reqno,amstex]{amsart}

\usepackage{epsfig}
\usepackage{amsmath}
\usepackage{amssymb}
\usepackage{amscd}

\usepackage[colorlinks=true,citecolor=blue]{hyperref}
\usepackage{stmaryrd}
\usepackage{epsfig}
\usepackage{amsmath,bm}
\usepackage{amssymb}
\usepackage{amssymb,amsthm,amsmath}
\usepackage{mathrsfs}
\usepackage{amscd}
\usepackage{graphicx}
\usepackage{pstricks}
\usepackage{theoremref}

\numberwithin{equation}{section}

\theoremstyle{plain}
\newtheorem{thm}{Theorem}[section]
\newtheorem{lem}[thm]{Lemma}

\theoremstyle{definition}

\newcommand{\N}{\mathbb{N}}

\begin{document}
	\title{On variational principles of  metric mean dimension  on subset in Feldman-Katok metric}

	\author{Kunmei Gao, Ruifeng Zhang}

	\address[K.~Gao] {School of Mathematics, Hefei University of Technology, Hefei, Anhui, 230009, P.R. China}
	\email{2102379003@qq.com}
	
	\address[ Rui.~Zhang] {School of Mathematics, Hefei University of Technology, Hefei, Anhui, 230009, P.R. China}
   \email{rfzhang@hfut.edu.cn}
	
	
	\begin{abstract}
	 In this paper, we studied the  metric mean dimension in Feldman-Katok(FK for short) metric.	 We introduced the notions of FK-Bowen metric mean dimension and  FK-Packing metric mean dimension on subset. And we established two variational principles.

	\end{abstract}

	\date{\today}
	
	\maketitle
	
	\section{Introduction}

	Let $(X,T)$ be a topological dynamical system (TDS for short), if $X$ is a compact metric space with a metric $d$, $T:\:X\to X$ is a continuous map. We call $M(X)$ the set of all Borel probability measures on $X$.

	A fundamental problem in ergodic theory is to classify the measure-preserving system (MPS for short) up to isomorphism. In 1958, Kolmogorov\cite{Kolmogrov} introduced the concept of entropy in ergodic theory, proving that entropy is an isomorphic invariant for MPSs. A notable achievement in the isomorphic problem is Ornstein's\cite{Ornstein} theory. He proved that any two equal entropy Bernoulli processes are isomorphic. The concept of a finitely determined process plays an important role in Ornstein's theory, and its definition is based on Hamming distance $\overline{d}_n$:\[\overline{d}_n(x_0x_1...x_{n-1},y_0y_1...y_{n-1})=\frac{\mid\{0\leq i\leq n-1:x_i\neq y_i\}\mid}{n}.\]
	By changing the Hamming distance $\overline{d} _n$ in Ornsten's theory to the edit distance $\overline{f} _n$:\[\overline{f}_n(x_0x_1...x_{n-1},y_0y_1...y_{n-1})=1-\frac{k}{n},\]where k is the largest inter such that there exist\[0\leq i_1\leq ...\leq i_k\leq n-1,0\leq j_1\leq ...\leq j_k\leq n-1\]and $x_{i_s}=y_{j_s}$ for $s=1,...,k$, Feldman\cite{Feldman} defined the  loosely Bernoulli system, 
	 and brought a new idea into the classification of MPSs. Based on this, a new theory which is parallel to Ornstein’s theory was then established \cite{Katok77,ORW,Feldman}.
	
The Feldman-Katok metric was introduced in \cite{KL}, which is a topological counterpart of the edit distance. In \cite{GK}, the authors used the Feldman-Katok metric to characterize zero entropy loosely Bernoulli MPSs, they also studied FK continuous and FK sensitive and obtained the Auslander-Yorke dichotomies. In 2021, Cai and Li\cite{CL}  study the entropy in the Feldman-Katok metric and  obatined an entropy formulae. Recently, Nie and Huang \cite{NH} further study the restricted sensitivity and entropy in FK metirc and obtained conditional entropy formulae.

Mean dimension was first introduced by Gromov\cite{Gromov} in 1999.  Then Lindenstrauss and Weiss\cite{LW} defined a metric version of mean dimension, called metric mean dimension. Mean dimension can be applied to solve embedding problems in dynamical systems (see for example, \cite{Gutman,L99, LT14}) and 
is also a meaningful quantity to describe system with infinite entropy. In dynamical systems, people are interested in the relation of the topological concepts and the measure-theoretic concepts. Between topological entropy and measure-theoretical entropy, there exists a variational principles. It is natural to ask wether there exists some verison of varianton relation for mean dimension. And this is not simple, the first variational principle for metric mean dimensions was  established by Lindenstrauss and Tuskmoto \cite{LT18} until 2017. After that, there are sequences of  researches  on variation principles of mean dimension, see \cite{CCL,CDZ,CLS,LT19,S,W,YCZ22packing} for example.

 In this paper,  we  study the metric mean dimension in Feldman-Katok metric. More precisely, we first introduced the notions of FK-Bowen metric mean dimension and FK-Packing metric mean dimension on subsets. Then we  established the  variational principle for the FK-Bowen  metric mean dimension (FK-Packing  metric mean dimension, respectively)  on subsets. The main results are the following twe theorems.
	
	\begin{thm}\label{thm-main}
	 Let $(X,T)$ be a $TDS$ and $K$ be a non-empty compact subset of $X$, then
	 $$\overline{mdim}_{FK}^B(T,K,d)=\limsup_{\epsilon\rightarrow 0}\frac{\sup\{\underline{h}_{\mu}^{FK}(T,\epsilon):\mu(K)=1,\mu\in M(X)\}. 	}{\log\frac{1}{\epsilon}}.$$	 	

	\end{thm}

\medskip	
	
		\begin{thm}\label{thm-packing}
		Let $(X,T)$ be a $TDS$ and $K$ be a non-empty compact subset of $X$, then $$\overline{mdim}_{FK}^P(T,K,d)=\limsup_{\epsilon\to0}\frac{\sup\{\overline{h}_{\mu}^{FK}(T,\epsilon):\mu(K)=1,\mu\in M(X)\}}{\log\frac{1}{\epsilon}}.$$
	\end{thm}	
		
		The rest of this paper is organized as follows. In Section 2, we introduce the notions of FK-Bowen metric mean dimension  and FK-Packing metric mean dimension on subsets and also the measure-theoretical local entropy in FK metric.  In Section 3, we prove the variational principle for FK-Bowen mean dimension, i.e. Theorem \ref{thm-main}. In Section 4, we prove the variational principle for FK-Bowen mean dimension, i.e. Theorem \ref{thm-packing}.

\subsection*{Acknowledgments} 	
 The authors are supported by NNSF of China (Grant No.~11871188 and ~12031019).
		
		\section{ Prelminary}
		
		In this section, we will introduce some notions and properties that will be used.
		
		\subsection{Feldman-Katok  metric }

		Let $(X,T)$ be a $TDS$ with $d$ be the metric on $X$. For $x,y\in X,\:\delta>0$ and $n\in\mathbb{N}$, we define an $(n,\delta)-match\ of\ x\ and\ y$ is an $\textbf{order preserving}(i.e\ \pi(i)<\pi(j),\:i<j)$ bijection$\pi:D(\pi)\to R(\pi)$ such that $D(\pi),R(\pi)\subset\{0,1,...,n-1\}$ and for every $i\in D(\pi)$ we have $d(T^ix,T^{\pi(i)}y)<\delta$. Let  $\mid\pi\mid$ be the cardinality of $D(\pi)$. 	
		 The FK metric defined on $X$ is\[d_{FK_n}(x,y):=\inf{\{\delta>0,\:\bar{f}_{n,\delta}(x,y)<\delta\}},\]
		
		where \[\bar{f}_{n,\delta}(x,y)=1-\frac{\max\{\mid\pi\mid:\pi\ is(n,\delta)-match\ of\ x\ and\ y\}}{n}.\]

		For each $\epsilon>0$, under the FK metric we can define the FK open ball and closed ball with a radius of $\epsilon$ as\[B_{FK_n}(x,\epsilon)=\{y\in X:d_{FK_n}(x,y)<\epsilon\},\]\[\overline{B}_{FK_n}(x,\epsilon)=\{y\in X:d_{FK_n}(x,y)\leq\epsilon\}.\]
		
	    For $x,y \in X$ and $n \in \N$, let
	    $$d_n(x,y)=\max_{0 \le i \le n-1}d(T^ix,T^iy) \text{ \ and \  }
	    \overline{d}_n(x,y)=\frac{1}{n}\sum_{i=0}^{n-1}d(T^ix,T^iy)$$
	    We have the following result.
	    \begin{lem}\cite[Lemma 2.3]{CL}
	    $d_{FK_n}(x,y) \le (\overline{d}_n(x,y))^{\frac{1}{2}}$ and $ \overline{d}_n(x,y)\le d_n(x,y)$.
	   \end{lem}

		\subsection{FK  metric mean dimension on subsets}
	    In this subsection, we will introduce two kinds of  metric mean dimension on subsets.
		Let $Z\subset X$, $\epsilon>0$, $N\in\mathbb{N}$, $s\in\mathbb{R}$. 
		We define\[M_{FK}(T,d,Z,s,N,\epsilon)=\inf\{\sum\limits_{i\in I}e^{-n_{i}s}\},\]where infimum takes all the finite or countable covers $\{B_{FK_{n_i}}(x_i,\epsilon)\}_{i\in I}$ of $Z$ and $n_i\geq N$, $x_i\in X$.
		
		We can find the quantity $M_{FK}(T,d,Z,s,N,\epsilon)$ dose not decrease as $N$ increases, so define\[M_{FK}(T,d,Z,s,\epsilon)=\lim_{N\to\infty}M_{FK}(T,d,Z,s,N,\epsilon)\]
		
		We can find that the quantity $M_{FK}(T,d,Z,s,\epsilon)$ has a critical value of parameter $s$ jumping from $\infty$ to $0$. We define such critical value as
		\begin{equation*}
			\begin{split}	
				M_{FK}(T,d,Z,\epsilon):&=\inf\{s:M_{FK}(T,d,Z,s,\epsilon)=0\}\\
				&=\sup\{s:M_{FK}(T,d,Z,s,\epsilon)=+\infty\}.
			\end{split}
		\end{equation*}
		
		Put\[\overline{mdim}_{FK}^B(T,Z,d)=\limsup\limits_{\epsilon\to0}\frac{M_{FK}(T,d,Z,\epsilon)}{\log\frac{1}{\epsilon}}\]
		
		We call the number $\overline{mdim}_{FK}^B(T,Z,d)$ as FK-Bowen metric mean dimension of $T$ on the set $Z$. 
		
		Next we introduce the FK-Packing metric mean dimension on subset.		
       Let $Z\subset X$, $\epsilon>0$, $N\in\mathbb{N}$, $s\in\mathbb{R}$. Define\[P_{FK}(T,d,Z,s,N,\epsilon)=\sup\{\sum\limits_{i\in I}e^{-n_{i}s}\},\]where supremum takes all the finite or countable pariwise disjoint closed families  $\{\overline{B}_{FK_{n_i}}(x_i,\epsilon)\}_{i\in I}$ of $Z$ and $n_i\geq N$, $x_i\in X$.
		
		We can find the quantity $P_{FK}(T,d,Z,s,N,\epsilon)$ does not increase as $N$ increases, so define\[P_{FK}(T,d,Z,s,\epsilon)=\lim\limits_{N\to\infty}P_{FK}(T,d,Z,s,N,\epsilon)\]Let\[\mathcal{P}_{FK}(T,d,Z,s,\epsilon)=\inf\{\sum\limits_{i=1}^{\infty}P(T,d,Z_i,s,\epsilon):\cup_{i\geq1}Z_i\supseteq Z\}\]
		
		We can find that the quantity $\mathcal{P}_{FK}(T,d,Z,s,\epsilon)$ has a critical value of parameter $s$ jumping from $\infty$ to $0$. We define such critical value as
		\begin{equation*}
		\begin{split}	
		\mathcal{P}_{FK}(T,d,Z,\epsilon):&=\inf\{s:\mathcal{P}_{FK}(T,d,Z,s,\epsilon)=0\}\\
		&=\sup\{s:\mathcal{P}_{FK}(T,d,Z,s,\epsilon)=+\infty\}.
		\end{split}
		\end{equation*}
		
		Put\[\overline{mdim}_{FK}^P(T,Z,d)=\limsup\limits_{\epsilon\to0}\frac{\mathcal{P}_{FK}(T,d,Z,\epsilon)}{\log\frac{1}{\epsilon}}\]
		
		We call the number $\overline{mdim}_{FK}^P(T,Z,d)$ as FK-Packing  metric mean dimension of $T$ on the set $Z$.
	
		\hspace*{\fill}\
		\subsection{Measure-theoretical local entropy in FK metric}
		Let $\mu \in M(X)$,  
		following the idea of Brin and Katok, we define the following quantities\[\overline{h}_{\mu}^{FK}(T,\epsilon):=\int\overline{h}_{\mu}^{FK}(T,x,\epsilon)d\mu,\]\[\underline{h}_{\mu}^{FK}(T,\epsilon):=\int\underline{h}_{\mu}^{FK}(T,x,\epsilon)d\mu,\]
		
		where\[\overline{h}_{\mu}^{FK}(T,x,\epsilon)=\limsup_{n\to\infty}-\frac{\log\mu(B_{FK_n}(x,\epsilon))}{n},\]\[\underline{h}_{\mu}^{FK}(T,x,\epsilon)=\liminf_{n\to\infty}-\frac{\log\mu(B_{FK_n}(x,\epsilon))}{n}.\]

		\subsection{Weighted FK-Bowen  metric mean dimension}
		
		For any function $f\in C(X,R)$, $f:X\to [ 0,+\infty)$, $s\geq0$, $N\in\mathbb{N}$ and $\epsilon>0$, define \[W_{FK}(T,f,d,X,s,N,\epsilon)=inf\{\sum\limits_{i\in I}c_ie^{-n_{i}s}\},\]where infimum takes all the finite or countable covers $\{(B_{FK_{n_i}}(x_i,\epsilon),c_i)\}_{i\in I}$, such that $0<c_i<\infty$, $x_i\in X$, $n_i\geq\mathbb{N}$ and \[\sum_{i\in I}c_i\chi_{B_i}\geq f,\]where $B_i=B_{FK_{n_i}}(x_i,\epsilon)$, and $\mathcal{X}$ denotes the characteristic function of $B_i$.
		
		For $Z\subset X$, $f=\chi_Z$, we set $W_{FK}(T,d,Z,s,N,\epsilon)=W_{FK}(T,\chi_Z,d,X,s,N,\epsilon)$. We can find the quantity $W_{FK}(T,d,Z,s,N,\epsilon)$ dose not decrease as $N$ increases, so define\[W_{FK}(T,d,Z,s,\epsilon)=\lim_{N\to\infty}W_{FK}(T,d,Z,s,N,\epsilon)\]
		
		We can find that the quantity $W_{FK}(T,d,Z,s,\epsilon)$ has a critical value of parameter $s$ jumping from $\infty$ to $0$. We define such critical value as
		\begin{equation*}
			\begin{split}	
				W_{FK}(T,d,Z,\epsilon):&=\inf\{s:W_{FK}(T,d,Z,s,\epsilon)=0\}\\
				&=\sup\{s:W_{FK}(T,d,Z,s,\epsilon)=+\infty\}.
			\end{split}
		\end{equation*}
		
		Put\[\overline{Wmdim}_{FK}^B(T,Z,d)=\limsup\limits_{\epsilon\to0}\frac{W_{FK}(T,d,Z,\epsilon)}{\log\frac{1}{\epsilon}}\]
		
		We call the number $\overline{Wmdim}_{FK}^B(T,Z,d)$ as weighted FK-Bowen  metric mean dimension.

		\section{Proof of Theorem \ref{thm-main}}
		 We first introduce some lemmas.		
		The following lemma is the famous $5r$-covering lemma.
		
		\begin{lem}\label{lem-5r} Let $(X,d)$ is a compact metric space, $\mathcal{B}=\{B(x_i,r_i)\}_{i\in I}$ be a family of closed (or open) balls in $X$. Then there exists a finite or countable subfamily $\mathcal{B}^{'}=\{B(x_i,r_i)\}_{i\in I^{'}}$ of pairwise disjoint balls in $\mathcal{B}$ such that\[\mathop{\cup}\limits_{B\in\mathcal{B}}B\subset \mathop{\cup}\limits_{i\in I^{'}}B(x_i,5r_i).\]
		\end{lem}
	\begin{proof}
		See Theorem 2.11 in \cite{Mattila}.	
	\end{proof}	
		
	The next lemma shows the relation between the FK metric mean dimension and the weighted FK metric mean dimension.
		
	 \begin{lem}\label{lem-weight}
	 	Let $Z\subset X$, $s\geq0$, $\epsilon,\delta>0$, we have\[M_{FK}(T,d,Z,s+\delta,N,6\epsilon)\leq W_{FK}(T,d,Z,s,N,\epsilon)\leq M_{FK}(T,d,Z,s,N,\epsilon),\]when $N$ is large enough, and there is\[\overline{mdim}_{FK}^B(T,Z,d)=\overline{Wmdim}_{FK}^B(T,Z,d).\]
	\end{lem}	
\begin{proof}.
	In the proof, we will follow Feng and Huang's method \cite{FH}.
	
	Let $Z\subset X$, $s\geq0$, $\epsilon,\delta>0$, $f=\chi_Z$, $c_i=1$. From the definition, we have $W_{FK}(T,d,Z,s,N,\epsilon)\leq M_{FK}(T,d,Z,s,N,\epsilon)$, for any $N\in\mathbb{N}$. 
	
	Next we will we will prove
	\begin{equation}\label{eq1}
	 M_{FK}(T,d,Z,s+\delta,N,6\epsilon)\leq W_{FK}(T,d,Z,s,N,\epsilon),
	\end{equation}
	when $N$ is large enough.

	We choose $N>2$ such that $n^2e^{-n\delta}<1$ for $n\geq N$. Let $\{(B_{FK_{n_i}}(x_i,\epsilon),c_i\}_{i\in I}$ be a family so that $I\subset\mathbb{N}$, $x_i\in X$, $0<c_i<\infty$, $n_i\geq N$ and 
	\begin{equation}\label{eq2}
    \sum_{i\in I}c_i\chi_{B_i}\geq\chi_Z,
	\end{equation} where $B_i=B_{FK_{n_i}}(x_i,\epsilon)$.
    Next we will show 
    \begin{equation}\label{eq3}
	 M_{FK}(T,d,Z,s+\delta,N,6\epsilon)\leq \sum_{i\in I}c_ie^{-s n_i},
	\end{equation} 
	 which implies that \eqref{eq1} holds.
	
	Let $I_n=\{i\in I:n_i=n\}$, $I_{n,k}=\{i\in I_n:i\leq k\}$ for $n\geq N$ and $k\in\mathbb{N}$. For simplicity we denote $B_i:=B_{FK_{n_i}}(x_i,\epsilon)$ and $5B_i:=B_{FK_{n_i}}(x_i,5\epsilon)$ for $i\in I$. We may assume $B_i\neq B_j$ when $i\neq j$. For any $t>0$, let
	\[Z_{n,t}=\{x\in Z:\sum_{i\in I_n}c_i\chi_{B_i}(x)>t\}\text{ \ and  \ } Z_{n,k,t}=\{x\in Z:\sum_{i\in I_{n,k}}c_i\chi_{B_i}(x)>t\}.\]
	
	We divide the proof of \eqref{eq3} into the following three steps.
	
	\hspace*{\fill}\
	
	\textbf{Step1:} For each  $n\geq N$, $k\in\mathbb{N}$, and $t>0$, there exists a finite set $\mathcal{J}_{n,k,t}\subset I_{n,k}$ such that the balls $B_i(i\in\mathcal{J}_{n,k,t})$ are pairwise disjoint, $Z_{n,k,t}\subset\cup_{i\in\mathcal{J}_{n,k,t}}5B_i$ and
	\begin{equation}\label{eq4}
	\#(\mathcal{J}_{n,k,t})e^{-sn}\leq\frac{1}{t}\sum_{i\in I_{n,k}}c_ie^{-sn}.
	\end{equation}
	
	Now we start to prove the above result. Since $I_{n,k}$ is finite, by approximating the $c_i$'s from above, we may assume that each $c_i$ is a positive rational.  By mulitiplying with common denominator, we may further assume every $c_i$ is a positive integer. Let $m$ be the smallest integer of $m\geq t$. Let $\mathcal{B}=\{B_i:i\in I_{n,k}\}$ and we define $v:\mathcal{B}\to\mathbb{Z}$ by $v(B_i)=c_i$. Then we can inductively define the integer-valued functions $v_0,...,v_m$ on $\mathcal{B}$ and subfamilies $\mathcal{B}_1,...,\mathcal{B}_m$ of $\mathcal{B}$ starting with $v_0=v$. Using Lemma \ref{lem-5r}(in which we take the
	metric $d_{FK_n}$ instead of $d$) we find  a pairwise disjoint subfamily $\mathcal{B}_1$ of $\mathcal{B}$ such that $$Z_{n,k,t}\subset \cup_{B\in\mathcal{B}}B\subset\cup_{B\in\mathcal{B}_1}5B.$$

	By repeatedly using Lemma \ref{lem-5r}, we can define disjoint subfamilies $\mathcal{B}_j$ of $\mathcal{B}$ inductively for $j=1,...,m$  such that
	\begin{equation} \label{eq5} 
	\mathcal{B}_j\subset\{B\in\mathcal{B}:v_{j-1}(B)\geq 1\},\:Z_{n,k,t}\subset\cup_{B\in\mathcal{B}_j}5B,
	\end{equation}
	and the function $v_j$ such that
	\begin{equation}\label{eq6}
		v_j(B)=\left\{\begin{array}{ll} v_{j-1}(B)-1,B\in\mathcal{B}_j;\\v_{j-1}(B),B\in \mathcal{B} \setminus\mathcal{B}_j.
	\end{array}\right.
	\end{equation}
	
    For every $j<m$, since every $x \in Z_{n,k,t}$ belongs to some 
    	$B\in \mathcal{B} $ with $\nu_j(B) \ge 1$, we have 
     \[Z_{n,k,t}\subset\bigg\{x:\sum_{B\in\mathcal{B},x\in B}v_j(B)\geq m-j \bigg\}.\]
	Hence \eqref{eq5} and \eqref{eq6} holds. Then
 
	\begin{eqnarray*}	
	\sum_{j=1}^{m}\#(\mathcal{B}_j)e^{-sn}
	&=&\sum_{j=1}^{m}\sum_{B\in\mathcal{B}_j}(v_{j-1}(B)-v_j(B))e^{-sn}\\
	&\leq& \sum_{B\in\mathcal{B}}\sum_{j=1}^{m}(v_{j-1}(B)-v_j(B))e^{-s n}\\
	&=&\sum_{B\in\mathcal{B}}(v_0(B)-v_m(B))e^{-s n}\\
	&\leq&\sum_{B\in\mathcal{B}}v(B)e^{-s n}=\sum_{i\in I_{n,k}}c_ie^{-s n}.
	\end{eqnarray*}
	
	Choose $j_0\in\{1,...,m\}$ so that $\#(\mathcal{B}_{j_0})$ is smallest. Then\[\#(\mathcal{B}_{j_0})e^{-s n}\leq\frac{1}{m}\sum_{i\in I_{n,k}}c_ie^{-s n}\leq\frac{1}{t}\sum_{i\in I_{n,k}}c_ie^{-sn}.\]
	
	Hence $\mathcal{J}_{n,k,t}=\{i\in I:B_i\in\mathcal{B}_{j_0}\}$ is as desired.
	
	\hspace*{\fill}\
	
	\textbf{Step 2: }For each $n\geq N$ and $t>0$, we have
	\begin{equation}\label{eq7}
	M_{FK}(T,d,Z_{n,t},s+\delta,N,\epsilon)\leq\frac{1}{n^2t}\sum_{i\in I_n}c_ie^{-sn}.
	\end{equation}
	
	To see this, we may assume $Z_{n,t}\neq\emptyset$. Since $Z_{n,k,t}\uparrow Z_{n,t}$, $Z_{n,k,t}\neq\emptyset$ when $k$ is large enough. Let  $\mathcal{J}_{n,k,t}$ be the set defined in Step 1. We define $E_{n,k,t}=\{x_i:i\in\mathcal{J}_{n,k,t}\}$. Note that the family of all non-empty subsets of $X$ is compact with respect to Hausdorff distance. It follows that there is a subsequence $\{k_j\}_{j=1}^{\infty}$ of natural numbers and a non-empty compact set $E_{n,t}\subset X$ such that $E_{n,k_j,t}\to E_{n,t}$ in the Hausdorff distance as $j\to\infty$. Since any two points in $E_{n,t}$ have a distance (with respect to $d_{FK_n}$) not less than $\epsilon$, (because of $E_{n,k,t}=\{x_i:i\in\mathcal{J}_{n,k,t}\}$ and $\mathcal{J}_{n,k,t}=\{i\in I:B_i\in\mathcal{B}_{j_0}\}$ and $B_i\neq B_j$ when $i\neq j$.) so do the points in $E_{n,t}$. Thus, $E_{n,t}$ is a finite set, moreover, $\#(E_{n,k_j,t})=\#(E_{n,t})$ when $j$ is large enough.
	 Hence
	 \begin{equation}
		\bigcup_{x\in E_{n,t}}B_{FK_n}(x,5.5\epsilon)\supset\bigcup_{x\in E_{n,k_j,t}}B_{FK_n}(x,5\epsilon)=\bigcup_{i\in\mathcal{J}_{n,k_j,t}}5B_i\supset Z_{n,k_j,t},
	 \end{equation}
      when $j$ is large enough, and thus $\cup_{x\in E_{n,t}}B_{FK_n}(x,6\epsilon)\supset Z_{n,t}$. Since $\#(E_{n,k_j,t})=\#(E_{n,t})$ when $j$ is large enough, we have
    \begin{equation}
    \#(E_{n,t})e^{-s n}\leq\frac{1}{t}\sum_{i\in I_n}c_ie^{-s n}.
	\end{equation}
	
	Then we have
	\begin{equation*}
	\begin{split}	
	M_{FK}(T,d,Z_{n,t},s+\delta,N,6\epsilon)&\leq\#(E_{n,t})e^{-(s+\delta)n}\\
	&\leq\frac{1}{e^{n\delta}t}\sum_{i\in I_n}c_ie^{-s n}\leq\frac{1}{n^2t}\sum_{i\in I_n}c_ie^{-s n}
	\end{split}
	\end{equation*}
	
	\hspace*{\fill}\
	
	\textbf{Step 3:} For any $t\in(0,1)$, we have \[M_{FK}(T,d,Z,s+\delta,N,6\epsilon)\leq\frac{1}{t}\sum_{i\in I}c_ie^{-s n_i}\]
	
	To see this, fix $t\in(0,1)$. Note that $\sum_{n=N}^{\infty}n^{-2}<1$. It follows that $Z\subset\cup_{n=N}^{\infty}Z_{n,n^{-2}t}$ from \eqref{eq2}. Then by \eqref{eq7} and $M_{FK}(T,d,\cdot,s+\delta,N,6\epsilon)$ is an outer measure, we have
	\begin{equation*}
	\begin{split}	
	M_{FK}(T,d,Z,s+\delta,N,6\epsilon)&\leq\sum_{n=N}^{\infty}M_{FK}(T,d,Z_{n,n^{-2}t},s+\delta,N,6\epsilon)\\
	&\leq\sum_{n=N}^{\infty}\frac{1}{t}\sum_{i\in I_n}c_ie^{-s n} =\frac{1}{t}\sum_{i\in I}c_ie^{-s n_i}.
	\end{split}
	\end{equation*}
	Let $t \rightarrow 1$, then \eqref{eq3} holds, and we finish the proof.
\end{proof}

\medskip
Next, we will prove an Frostman's lemma in FK metric.
\begin{lem} \label{lem-frostman}
	Let $K$ be a non-empty compact subset of $X$ and $s\geq 0$, $N\in\mathbb{N}$, $\epsilon>0$.  Suppose that $c:=W_{FK}(T,d,K,s,N,\epsilon)>0$. Then there exists a Borel probability measure $\mu\in M(X)$ such that $\mu(K)=1$ and\[\mu(B_{FK_n}(x,\epsilon))\leq\frac{1}{c}e^{-sn}\]holds for all $x\in X$, $n\geq N$.
\end{lem}
	\begin{proof}
		Clearly $c<\infty$. We define a function $p$ on the space $C(X)$, $C(X)$ is a space consisting of all continuous real-valued functions on $X$.\[p(f)=\frac{1}{c}W_{FK}(T,\mathcal{X}_K\cdot f,d,X,s,N,\epsilon)\]
		
		Let $\textbf{1}\in C(X)$ denote the constant function $\textbf{1}(x)\equiv1$. It is easy to verify that 
		\\(1)$p(f+g)\leq p(f)+p(g)$ for any $f,\:g\in C(X).$
		\\(2)$p(tf)=tp(f)$ for any $t\geq0$ and $f\in C(X).$
		\\(3)$p(\textbf{1})=1,\:0\leq p(f)\leq\parallel f\parallel_{\infty}$ for any $f\in C(X)$, and $p(g)=0$ for $g\in C(X)$ with $g\leq0.$
		
		By the Hahn-Banach Theorem, we can extend the linear functional $t\mapsto tp(\textbf{1})$, $t\in\mathbb{R}$, from the subspace of the constant function to a linear functional $L:C(X)\to\mathbb{R}$ satisfying\[L(\textbf{1})=p(\textbf{1}) \:\:and\:\:-p(-f)\leq L(f)\leq p(f)\:\:for\:any\:f\in C(X).\]
		
		If $f\in C(X)$ with $f\geq0$, then $p(-f)=0$ and so $L(f)\geq 0$. Hence, combining the fact $L(\textbf{1})=1$, we can use the the Riesz Representation Theorem to find a Borel probability measure $\mu$ on X such that \[L(f)=\int fd\mu,\:\:\:\forall f\in C(X).\]
		
		Now, we show that $\mu(K)=1$. We see for any compact set $E\subset X\cap K$, by Uryson lemma there is $f\in C(X)$ such that $0\leq f\leq1$, $f(x)=1$ for $x\in E$ and $f(x)=0$ for $x\in K$. Then $f\cdot\chi_K\equiv0$ and thus $p(f)=0$. Hence $\mu(E)\leq L(K)\leq p(f)=0$. This shows $\mu(X\setminus K)=0$, i.e.$\mu(K)=1$.
		
		Following, we prove that\[\mu(B_{FK_n}(x,\epsilon))\leq\frac{1}{c}e^{-s n},\:\:\forall x\in X,\:n\geq N.\]
		
		We can find that for any compact set $E\subset B_{FK_n}(x,\epsilon)$, by Uryson lemma, there exists $f\in C(X)$, such that $0\leq f\leq1$, $f(y)=1$ for $y\in E$ and $f(y)=0$ for $y\in X\setminus B_{FK_n}(x,\epsilon)$. Then $\mu(E)\leq L(f)\leq p(f)$. Since $f\cdot\chi_K\leq\chi_{\mathcal{B}_{FK_n}(x,\epsilon)}$ and $n\geq N$, we have\[W_{FK}(T,\chi_K\cdot f,d,X,s,\epsilon)\leq e^{-s n},\]and thus $p(f)\leq \frac{1}{c}e^{-sn}$. Therefore,\[\mu(E)\leq\frac{1}{c}e^{-sn}.\]It follows that
		\begin{equation*}
		\begin{split}	
		\mu(B_{FK_n}(x,\epsilon))&=\sup\{\mu(E):E\:is\:a\:compact\:subset\:of\:B_{FK_n}(x,\epsilon)\}\\
		&\leq\frac{1}{c}e^{-s n}.
		\end{split}
		\end{equation*}
	\end{proof}

	    Using the above lemmas above, we can  prove Theorem \ref{thm-main}.
		\begin{proof}
		 We first prove
	$$\overline{mdim}_{FK}^B(T,K,d) \ge \limsup_{\epsilon\rightarrow 0}\frac{\sup\{\underline{h}_{\mu}^{FK}(T,\epsilon):\mu(K)=1,\mu\in M(X)\}. 	}{\log\frac{1}{\epsilon}}.$$

	 For any $\mu\in M(X)$, $\mu(K)=1$,  and $\epsilon>0$,  we only need to prove
	$$M_{FK}(T,d,K,\frac{\epsilon}{2})\geq \underline{h}_{\mu}^{FK}(T, \epsilon) =\int\underline{h}_{\mu}^{FK}(T,x,\epsilon)d\mu.$$
		
		Fix $\epsilon>0$ and $l\in\mathbb{N}$, let
		$$\alpha=min\{l,\int\underline{h}_{\mu}^{FK}(T,x,\epsilon)d\mu-\frac{1}{l}\}.$$
		Then there exist a Borel set $A_l\subset X$ with $\mu(A_l)>0$ and $N\in\mathbb{N}$ such that
		$$\mu(B_{FK_n}(x,\epsilon))\leq e^{-\alpha n},\:x\in A_l,\:n\geq N.$$
		
		Now, let $\{B_{FK_{n_i}}(x_i,\frac{\epsilon}{2})\}_{i\in I}$ be a  countable or finite family so that $x_i\in X$, $n_i\geq N$ and $\cup_{i\in I}B_{FK_{n_i}}(x_i,\frac{\epsilon}{2})\supset K\cap A_l$. Let\[I_1=\{i\in I:B_{FK_{n_i}}(x_i,\frac{\epsilon}{2})\cap(A_l\cap K)\neq\emptyset\}.\]
		
		Choose $y_i\in B_{FK_{n_i}}(x_i,\frac{\epsilon}{2})\cap(A_l\cap K)$, then
		\begin{equation*}
			\begin{split}	
				\sum_{i\in I}e^{-\alpha n_i}
				&\geq \sum_{i\in I_1}e^{-\alpha n_i}\\
				&\geq \sum_{i\in I_1}\mu(B_{FK_{n_i}}(y_i,\epsilon))\geq
				\sum_{i\in I_1}\mu(B_{FK_{n_i}}(x_i,\frac{\epsilon}{2}))\\
				&\geq\mu(K\cap A_l)=\mu(A_l)>0
			\end{split}
		\end{equation*}
		
		It follows that\[M_{FK}(T,d,K,\alpha,N,\frac{\epsilon}{2})\geq M_{FK}(T,d,K\cap A_l,\alpha,N,\frac{\epsilon}{2})\geq\mu(A_l).\]
		
		Therefore\[M_{FK}(T,d,K,\frac{\epsilon}{2})\geq\alpha.\]
		
		Letting $l\to\infty$, we have\[M_{FK}(T,d,K,\frac{\epsilon}{2})\geq\int\underline{h}_{\mu}^{FK}(T,x,\epsilon)d\mu.\]
		
		Then\[\overline{mdim}_{FK}^B(T,K,d)\geq\limsup_{\epsilon\to0}\frac{1}{\log\frac{1}{\epsilon}}sup\{\underline{h}_{\mu}^{FK}(T,\epsilon):\mu(K)=1,\mu\in M(X)\}.\]
		
		\medskip
		
		Next we prove 
		$$\overline{mdim}_{FK}^B(T,K,d) \le \limsup_{\epsilon\rightarrow 0}\frac{\sup\{\underline{h}_{\mu}^{FK}(T,\epsilon):\mu(K)=1,\mu\in M(X)\}	}{\log\frac{1}{\epsilon}}.$$

		We assume that $\overline{mdim}_{FK}^B(T,K,d)>0$, by Lemma \ref{lem-weight}, we have \[\overline{mdim}_{FK}^B(T,K,d)=\overline{Wmdim}_{FK}^B(T,K,d).\]
		
		Let $0<\lambda<\overline{Wmdim}_{FK}^B(T,K,d)$. Then we can find a sequence $0<\epsilon_k<1$ that covergences to 0 as $k\to\infty$ so that\[\overline{Wmdim}_{FK}^B(T,K,d)=\lim\limits_{k\to\infty}\frac{W_{FK}(T,d,K,\epsilon_k)}{\log\frac{1}{\epsilon_k}}.\]
		
		Hence, fix a sufficiently large $k$ there is $N_0\in\mathbb{N}$ such that \[c:=W_{FK}(T,d,K,\lambda \log\frac{1}{\epsilon_k},N_0,\epsilon_k)>0,\] by Lemma \ref{lem-frostman}, there exists a Borel probability measure $\mu\in M(X)$ such that $\mu(K)=1$ and $$\mu(B_{FK_n}(x,\epsilon_k))\leq\frac{1}{c}e^{-\lambda \log\frac{1}{\epsilon_k}n},$$ holds for all $x\in X$, $n\geq N_0$.
		Hence 
		$$ \underline{h}_{\mu}^{FK}(T,x,\epsilon_k)=\liminf_{n \rightarrow \infty}-\frac{\mu(B_{FK_n}(x,\epsilon_k))}{n}\ge \lambda\log \frac{1}{\epsilon_k}, \forall x \in X ,$$
		and
		 \[\frac{\sup\{\underline{h}_{\mu}^{FK}(T,\epsilon_k),\mu\in M(X),\mu(K)=1\}}{\log\frac{1}{\epsilon_k}}\geq\lambda.\]
		Hence
		$$\limsup_{\epsilon\rightarrow 0}\frac{\sup\{\underline{h}_{\mu}^{FK}(T,\epsilon):\mu(K)=1,\mu\in M(X)\}. 	}{\log\frac{1}{\epsilon}}\ge \lambda.$$

	\end{proof}

		\section{Proofs of Theorem \ref{thm-packing}}

		 We first introduce some lemmas.

	\begin{lem}\label{flj}
		Let $Z\subset X$ and $s$, $\epsilon>0$. Assume $\mathcal{P}_{FK}(T,d,Z,s,\epsilon)=\infty$. Then for any given finite
		interval $(a, b)\subset\mathbb{R}$ with $a\geq0$ and any $N\in\mathbb{N}$, there exists a finite disjoint collection $\{\overline{B}_{FK_{n_i}}(x_i,\epsilon)\}$ such that $x_i\in Z$, $n_i\geq N$ and $\sum_ie^{-sn_i}\in(a,b)$.
	\end{lem}
	\begin{proof}
		Take $N_1>N$ large enough such that $e^{-sN_1}<b-a$. Since $\mathcal{P}_{FK}(T,d,Z,s,\epsilon)=\infty$, we have $\mathcal{P}_{FK}(T,d,Z,s,N_1,\epsilon)=\infty$. Thus, there is a finite disjoint collection $\{\overline{B}_{FK_{n_i}}(x_i,\epsilon)\}$ such that $x_i\in Z$, $n_i\geq N_1$ and $\sum_ie^{-sn_i}>b$. Since $e^{-sn_i}\leq e^{-sN_1}<b-a$, by discarding elements in this collection one by one until we can have $\sum_ie^{-sn_i}\in(a,b)$.
	\end{proof}

	\begin{lem}\label{lem-2}
		Let $Z \subset X$ and $s>0$. Then for any $ 0<\epsilon_1<\epsilon_2$,
		$$P_{FK}(T,d,\overline{Z},s, \epsilon_2)\le P_{FK}(T,d,Z,s, \epsilon_1).$$
	\end{lem}
	\begin{proof}
		Let  $N \in \mathbb{N}$ and $\{\overline{B}_{{FK}_{n_i}}(x_i,\epsilon_2)\}_{i \in I}$ be a family of pairwise disjoint balls with $x_i \in \overline{Z}$ and $n_i \ge N$ for any $i \in I$.
		
		If $x_i \in Z$ then let $y_i=x_i$ and we have $$\overline{B}_{{FK}_{n_i}}(y_i,\epsilon_1)=\overline{B}_{{FK}_{n_i}}(x_i,\epsilon_1)\subset \overline{B}_{{FK}_{n_i}}(x_i,\epsilon_2).$$
		If $x_i \in \overline{Z}\setminus Z$, then we can choose $y_i \in Z$ such that $d_{{FK}_{n_i}}(x_i,y_i)<\epsilon_2-\epsilon_1$. Then for any $z \in \overline{B}_{{FK}_{n_i}}(y_i, \epsilon_1)$, 
		$$d_{{FK}_{n_i}}(z,x_i) \le d_{{FK}_{n_i}}(z,y_i) +d_{{FK}_{n_i}}(y_i,x_i) \le \epsilon_1+\epsilon_2-\epsilon_1=\epsilon_2,$$
		hence we have
		$$\overline{B}_{{FK}_{n_i}}(y_i,\epsilon_1)\subset \overline{B}_{{FK}_{n_i}}(x_i,\epsilon_2).$$
		
		Since $\{\overline{B}_{{FK}_{n_i}}(x_i,\epsilon_2)\}_{i \in I}$ are pairwise disjoint, $\{\overline{B}_{{FK}_{n_i}}(y_i,\epsilon_1)\}_{i \in I}$ are pairwise disjoint with $y_i \in Z$ and $n_i \ge N$ for any $i \in I$. Hence
		$$P_{FK}(T,d,\overline{Z},N,s, \epsilon_2)\le P_{FK}(T,d,Z,N,s, \epsilon_1),$$ 
		which implies $$P_{FK}(T,d,\overline{Z},s, \epsilon_2)\le P_{FK}(T,d,Z,s, \epsilon_1).$$
		
	\end{proof}

	Recall that a set in a metric space is called {\it  analytic} if it is a continuous image of $\mathcal{N}$, where $\mathcal{N}$ is the set of infinite sequences of natural numbers. It is known that every Borel set is analytic(see Chapter 11 of \cite{J-settheory}).
	
	\begin{lem}\label{lem-packing-lower}
		Let $\epsilon>0$, $Z\subset X$ be analytic with $\mathcal{P}_{FK}(T,d,Z,\epsilon)>0$. For any $0<s<\mathcal{P}_{FK}(T,d,Z,\epsilon)$, there exists a compact set $K\subset Z$ and $\mu\in M(K)$ such that $\overline{h}_{\mu}^{FK}(T,\epsilon)\geq s$.
	\end{lem}
	\begin{proof}
		In the proof we will follow Feng and Huang's method in \cite{FH}.
		
		By the definition of  analytic set, there exists a continuous surjective map $\Psi :\mathcal{N} \rightarrow Z$. Let  $\Gamma_{n_{1},n_{2},\dots,n_{p}}=\{(m_{1},m_{2},\dots) \in \mathcal{N}:m_{1} \leq n_{1},m_{2} \leq n_{2},\dots,m_{p} \leq n_{p}\}$ be an element in $\mathcal{N}$,  we denote $Z_{n_{1},n_{2},\dots,n_{p}}=\Psi(\Gamma_{n_{1},n_{2},\dots,n_{p}})$.
		
		Choose $t \in (s,\mathcal{P}_{FK}(T,d,Z,\epsilon))$. 
		Next, we will inductively construct the following sequences:.
		\begin{itemize}
			\item[(S-1)] A sequence of finite set$\{K_{i}\}_{i=1}^{\infty}$ with $K_i \subset Z$.
			\item[(S-2)]A sequence of finite measures $(\mu_{i})_{i=1}^{\infty}$ such that for each $i$, $\mu_i$ is supported on $K_i$.
			\item[(S-3)]  A sequence of integers $\{n_i\}_{i=1}^{\infty}$ and a sequence of positive numbers $\{\gamma_{i}\}_{i=1}^{\infty}$.
			\item[(S-4)] A sequence  integer-valued maps $\{ m_{i}: K_{i} \rightarrow \mathbb{N}\}_{i=1}^{\infty}$.
		\end{itemize}
		
		And these sequences will satifiy the following conditions.
		
		\begin{itemize}
			\item[(C-1)] For each $i$, elements in $\mathcal{F}_i=\{\overline{B}(x,\gamma_i)\}_{x \in K_i}$ are pairwise disjoint. And each element in $\mathcal{F}_{i+1}$ is a subset of 
			$\overline{B}(x,\frac{\gamma_i}{2})$ for some $x \in K_i$.
			\item[(C-2)] For each $i$, $K_i \subset Z_{n_1,n_2,\cdots,n_i}$ and 
			$\mu_i=\sum_{y \in K_i}e^{-m_i(y)s}\delta_y$. And $1<\mu_1(K_1)<2$.
			\item[(C-3)] For each $x \in K_i$ and $z \in B(x,\gamma_i)$
			\begin{equation}\label{C-3-1}
			\overline{B}_{FK_{m_i(x)}}(z,\epsilon) \cap \bigcup_{y \in K_i \setminus \{x\}}\overline{B}(y,\gamma_i)=\emptyset
			\end{equation}
			and
			\begin{equation}\label{C-3-2}
			\mu_{i}(\overline{B}(x,\gamma_i))<\sum_{y \in E_{i+1}(x)}e^{-m_{i+1}(y)s}<(1+2^{-i-1})\mu_i(\overline{B}(x,\gamma_i))
			\end{equation}
			where $E_{i+1}(x)=\overline{B}(x,\frac{\gamma_i}{4})\cap K_{i+1}$.
		\end{itemize}
		
		\medskip
		
		Assume the sequences $K_{i}$, $\mu_{i}$, $m_{i}(\cdot)$, $n_{i}$ and $\gamma_{i}$ have been constructed.	Next we will construct a compact set $K\subseteq Z$ and a measure $\mu \in M(K)$ and show that $\overline{h}_{\mu}^{FK}(T,\epsilon)\geq s$.

		By \eqref{C-3-2}, for each $V_i \in \mathcal{F}_i$,
		$$\mu_i(V_i)\le \mu_{i+1}(V_i)=\sum_{V\in \mathcal{F}_{i+1},V\subset V_i}\mu_{i+1}(V)\le (1+2^{-(i+1)})\mu_i(V_i).$$
		Using the above inequalities repeatedly, we have for any $j>i$ and $V_i \in \mathcal{F}_i$,
		\begin{equation}\label{ineq-mui}
		\mu_{i}(V_{i}) \leq \mu_{j}(V_{i}) \leq \prod_{n=i+1}^{j}(1+2^{-n})\mu_{i}(V_{i}) \leq C\mu_{i}(V_{i}), 
		\end{equation}
		where $C:=\prod_{n=1}^{\infty}(1+2^{-n})<\infty$.
		
		\bigskip
		
		Let $\widetilde{\mu}$ be the limit point of $\mu_{i}$ in weak* topology. We denote
		$$K=\bigcap_{n=1}^{\infty}\overline{\bigcup_{i \geq n}K_{i}},$$
		then $\widetilde{\mu}$ is supported on $K$, moreover
		$$K=\bigcap_{n=1}^{\infty}\overline{\bigcup_{i \geq n}K_{i}} \subset \bigcap_{p=1}^{\infty}\overline{Z_{n_{1}n_{2}\dots n_{p}}}.$$
		Since $\Psi$ is continuous, we can obtain $$\bigcap_{p=1}^{\infty}\overline{Z_{n_{1},n_{2},\dots, n_{p}}}=\bigcap_{p=1}^{\infty}Z_{n_{1}n_{2}\dots n_{p}}\subset Z$$
		by applying Cantor's diagonal argument.
		In fact, let $\omega_ \in \bigcap_{p=1}^{\infty}\overline{Z_{n_{1},n_{2},\dots, n_{p}}}$. 
		For $p=1$, there exist $\{\omega_{k}^{1}\}_{k=1}^{\infty}\subset Z_{n_{1}}$
		such that $\omega_{k}^{1} \rightarrow \omega$, where $$\omega_{k}^{1}=\Psi((m_{1}^{k,1},m_{2}^{k,1},\dots,m_{p}^{k,1},\dots))$$ with 
		$m_{1}^{k,1} \le n_1$. 
		Then there exists $N_1 \in \mathbb{N}$ and $m_1 \le n_1$ such that when $k \ge N_1$, $m_{1}^{k,1}=m_1$. 
		For $p=2$, there exist $\{\omega_{k}^{2}\}_{k=1}^{\infty}\subset Z_{n_{1},n_{2}}$
		such that $\omega_{k}^{2} \rightarrow \omega$, where $$\omega_{k}^{2}=\Psi((m_1,m_{2}^{k,2},\dots,m_{p}^{k,2},\dots))$$ with 
		$m_{2}^{k,2} \le n_2$. 
		Then there exists $N_2>N_1 \in \mathbb{N}$ and $m_2 \le n_2$ such that when $k \ge N_2$, $m_{2}^{k,2}=m_2$.  Inductively for $p>2$,
		there exist $\{\omega_{k}^{p}\}_{k=1}^{\infty}\subset Z_{n_{1},n_{2},\dots, n_{p}}$ such that $\omega_{k}^{p} \rightarrow \omega$. 
		where $$\omega_{k}^{p}=\Psi((m_1,m_2,\dots,m_{p-1},m_{p}^{k,p},m_{p+1}^{k,p},\dots))$$ with 
		$ m_{p}^{k,p} \le n_p$. Then there exists $N_p>N_{p-1} \in \mathbb{N}$ and $m_p \le n_p$ such that when $k \ge N_p$, $m_{p}^{k,p}=m_p$. Let $\vec{m}=(m_1,m_2,m_3,\dots)$ then 
		$$\omega=\Psi(\vec{m})\in \bigcap_{p=1}^{\infty}Z_{n_{1}n_{2}\dots n_{p}}.$$
		Hence $K$ is a compact subset of $Z$.
		
		By \eqref{ineq-mui}, for any $x \in K_i$
		$$e^{-m_{i}(x)s}=\mu_{i}(\overline{B}(x,\gamma_{i})) \leq \widetilde{\mu}(\overline{B}(x,\gamma_{i}))\le C\mu_{i}(\overline{B}(x,\gamma_{i}))=Ce^{-m_{i}(x)s}.$$
		In particular, 
		$$1 \leq \sum_{x \in K_{1}}\mu_{1}(B(x,\gamma_{1})) \leq \widetilde{\mu}(K) \leq \sum_{x \in K_{1}}C\mu_{1}(B(x,\gamma_{1})) \leq 2C.$$
		Since $K \subseteq \bigcup_{x \in K_{i}}\overline{B}(x,\frac{\gamma_{i}}{2})$. then by \eqref{C-3-1}, for any $x \in K_{i}$, $z \in \overline{B}(x,\gamma_{i})$, we have
		$$\widetilde{\mu}(\overline{B}_{FK_{m_{i}(x)}}(z,\epsilon)) \leq \widetilde{\mu}(\overline{B}(x,\frac{\gamma_{i}}{2})) \leq Ce^{-m_{i}(x)s}.$$
		For each $z \in K$ and $i \in \mathbb{N}$, there exists some $x \in K_{i}$ such that  $z \in \overline{B}(x,\frac{\gamma_{i}}{2})$, hence
		$$\widetilde{\mu}(B_{FK_{m_{i}(x)}}(z,\epsilon)) \leq Ce^{-m_{i}(x)s}.$$
		
		Let $\mu=\frac{\widetilde{\mu}}{\widetilde{\mu}(K)}$, so $\mu \in M(K)$. Then  for each $z \in K$, there exist sequence $k_{i} \rightarrow \infty$, such that $\mu (B_{FK_{k_{i}}}(z,\epsilon)) \leq \frac{Ce^{-k_{i}s}}{\widetilde{\mu}(K)}$. Thus we have
		\begin{eqnarray*}
			\overline{h}_{\mu}^{FK}(T, \epsilon) &=& \int_{K}\overline{h}_{\mu}^{FK}(T,x,\epsilon)d\mu(x)\\
			&=& \int_{K}\limsup_{n \rightarrow \infty}-\frac{1}{n}\log(\mu(B_{FK_n}(x,\epsilon)))d\mu(x)\\
			&\geq& \int_{K}\limsup_{k_{i} \rightarrow \infty}-\frac{1}{k_i}\log \left(\frac{Ce^{-k_{i}s}}{\widetilde{\mu}(K)}\right)d\mu(x)\\
			&=& \int_{K}\limsup_{k_{i} \rightarrow \infty}-\frac{1}{k_{i}}(\log C+\log e^{-k_{i}s})d\mu(x)\\
			&=&s\mu(K)\\
			&=&s
		\end{eqnarray*}

		Now we start the construction, the construction is divided by three steps.

		{\noindent {\bf Step 1:}} Construction of $K_{1}$, $\mu_{1}$, $m_{1}(\cdot)$, $n_{1}$ and $\gamma_{1}$.
		
		Recall that we choose $t<\mathcal{P}_{FK}(T,d,Z,\epsilon)$. So $\mathcal{P}_{FK}(T,d,Z,t,\epsilon)=\infty$. Let
		$$H_{1}=\bigcup \{G \subset X: G \text{ is an open set}, \mathcal{P}_{FK}(T,d,Z \cap G,t,\epsilon)=0\}.$$
		By the separability of $X$, $H_1$ is a countable union of the open sets $G$'s. So
		$$Z \cap H_{1}=Z \cap (\cup_{i} G_{i})=\cup_i (Z \cap G_{i}),$$
		we obtain 
		$\mathcal{P}_{FK}(T,d,Z\cap H_1,t,\epsilon)\le \sum_{i}\mathcal{P}_{FK}(T,d,Z \cap G_i,t,\epsilon)=0$.

		Let 
		$Z'=Z \setminus H_{1}=Z \cap (X \setminus H_{1}).$
		For every open set $G \subset X$, either $Z' \cap G =\emptyset$ or $\mathcal{P}_{FK}(T,d,Z' \cap G,t,\epsilon)>0$. To see the conclusion, assume that $G$ is an open set with $\mathcal{P}_{FK}(T,d,Z' \cap G,t,\epsilon)=0$, notice that
		\begin{eqnarray*}
			Z\cap G
			&=&（(Z\cap(X\setminus H_1))\cup (Z\cap H_1))\cap G 
			\subset ( Z'\cap G) \cup (Z \cap H_1)
		\end{eqnarray*}
		then $\mathcal{P}_{FK}(T,d,Z \cap G,t,\epsilon) \leq \mathcal{P}_{FK}(T,d,Z' \cap G,t,\epsilon)  + \mathcal{P}_{FK}(T,d,Z \cap H_1,t,\epsilon) =0$, it implies $G \subset H_{1}$ and then $Z' \cap G =\emptyset$.
		
		Note that $$\mathcal{P}_{FK}(T,d,Z,t,\epsilon)  \leq \mathcal{P}_{FK}(T,d,Z',t,\epsilon)  +\mathcal{P}_{FK}(T,d,Z'\cap H_1,t,\epsilon) =\mathcal{P}_{FK}(T,d,Z',t,\epsilon), $$ hence $\mathcal{P}_{FK}(T,d,Z',t,\epsilon)=\mathcal{P}_{FK}(T,d,Z,t,\epsilon)=+\infty$. Since $t>s$, by definition we have  $\mathcal{P}_{FK}(T,d,Z',s,\epsilon)=+\infty$.
		
		By Lemma \ref{flj}, we can find finite set $K_{1} \subseteq Z'$ and integer-valued map $m_{1}$ on $K_{1}$ such that the collection $\{\overline{B}_{FK{m_{1}(x)}}(x,\epsilon)\}_{x \in K_{1}}$ is disjoint and $$\sum_{x \in K_{1}}e^{-m_{1}(x)s} \in (1,2).$$
		
		Let $$\mu_{1}=\sum_{x \in K_{1}}e^{-m_{1}(x)s}\delta_{x},$$ 
		where $\delta_{x}$ is the Dirac measure at $x$. By Lemma \ref{d-fk-mean}, we can choose $\gamma_{1}>0$ small enough such that for any function $z:K_{1} \rightarrow X$ with $d(x,z(x))<\gamma_{1}$, we have for every $x \in K_{1}$,
		\begin{equation}\label{step1-disjoint}
		(\overline{B}(z(x),\gamma_{1})) \cup (\overline{B}_{FK_{m_{1}(x)}}(z(x),\epsilon)) \cap \left(\bigcup_{y \in K_{1} \setminus \{x\}}\overline{B}(z(y),\gamma_{1}) \cup \overline{B}_{FK_{{m_{1}(y)}}}(z(y),\epsilon)\right)=\emptyset
		\end{equation}
		Since $K_{1} \subseteq Z'$, for each $x \in K_{1}$,
		$$\mathcal{P}_{FK}(T,d,Z \cap B(x,\frac{\gamma_{1}}{4}),t,\epsilon) \geq \mathcal{P}_{FK}(T,d,Z' \cap B(x,\frac{\gamma_{1}}{4}),t,\epsilon) >0.$$ Hence we can find a large $n_{1} \in \mathbb{N}$  such that $K_{1} \subseteq Z_{n_{1}}$, and
		$$\mathcal{P}_{FK}(T,d,Z_{n_1} \cap B(x,\frac{\gamma_{1}}{4}),t,\epsilon)>0$$
		for each $x \in K_{1}$.
		
		\medskip
		
		{\noindent \bf{Step 2:}} Construction of $K_{2}$, $\mu_{2}$, $m_{2}(\cdot)$, $n_{2}$ and $\gamma_{2}$.
		
		By \eqref{step1-disjoint},  we know $\{\overline{B}(x,\gamma_{1})_{x \in K_{1}}\}$ are pairwise disjoint. Since for each each $x \in K_{1}$, 	$$\mathcal{P}_{FK}(T,d,Z_{n_1} \cap B(x,\frac{\gamma_{1}}{4}),t,\epsilon)>0,$$
		similar to what we did in step 1, 
		we can choose a finite set
		$$E_{2}(x) \subseteq Z_{n_{1}} \cap B \left(x,\frac{\gamma_{1}}{4} \right),$$
		and a integer-valued map
		$$m_{2}:E_{2}(x) \rightarrow \mathbb{N} \bigcap [\max\{m_{1}(y):y \in K_{1}\},\infty)$$
		such that
		
		(2-$a$)$\mathcal{P}_{FK}(T,d,Z_{n_{1}} \cap G,t, \epsilon)>0$, where $G$ is a open set with $G \cap E_{2}(x) \neq \emptyset$;
		
		(2-$b$)The elements in $\{\overline{B}_{FK_{m_{2}(y)}}(y,\epsilon)\}_{y \in E_{2}(x)}$ are disjoint, and
		$$\mu_{1}(\{x\})=e^{-m_1(x)s}<\sum_{y \in E_{2}(X)}e^{-m_{2}(y)s}<(1+2^{-2})\mu_{1}(\{x\}).$$
		
		More precisely, we fix $x \in K_{1}$ and denote $F_{1,x}=Z_{n_{1}} \cap B\left(x,\frac{\gamma_{1}}{4}\right)$. Let
		$$H_{2,x}:= \bigcup \{G \subseteq X :G \text{ is an open set},\mathcal{P}_{FK}(T,d, F_{1,x} \cap G,t,\epsilon)=0\}.$$
		Let $F_{1,x}'=F_{1,x} \setminus H_{2,x}$, then similarly as in step 1, we can show $$\mathcal{P}_{FK}(T,d,F_{1,x}',t, \epsilon)=\mathcal{P}_{FK}(T,d,F_{1,x},t, \epsilon)>0.$$
		Moreover, $\mathcal{P}_{FK}(T,d, F_{1,x}' \cap G,t, \epsilon)>0$ for any open set $G$ with $F_{1,x}' \cap G \neq \emptyset$. 
		
		Recall that $s<t$,  then 
		$\mathcal{P}_{FK}(T,d, F_{1,x}',t, \epsilon)=+\infty$. By Lemma \ref{flj}, we can find a finite set $E_{2}(x) \subseteq F_{1,x}'$ and a integer-valued map 
		$$m_{2}:E_{2}(x) \rightarrow \mathbb{N} \bigcap [\max\{m_{1}(y):y \in K_{1}\},\infty)$$ 
		such that (2-$b$) holds. 
		Notice that $E_2(x) \subset F'_{1,x}$, if $G$ is an open set with $G\cap E_2(x) \neq \emptyset$, then $G\cap F'_{1,x}\neq \emptyset$, hence
		$$\mathcal{P}_{FK}(T,d, Z_{n_1} \cap G,t,\epsilon) \ge \mathcal{P}_{FK}(T,d, F_{1,x}' \cap G,t,\epsilon)>0,$$
		so (2-$a$) holds.
		
		Since the balls in  $\{\overline{B}(x,\gamma_{1})_{x \in K_{1}}\}$ are pairwise disjoint, for any $x,x' \in K_{1}$ with $x \neq x'$, we have $E_{2}(x) \cap E_{2}(x')=\emptyset$. We can construct a finite set  $K_2$ with
		$$K_{2}=\bigcup_{x \in K_{1}}E_{2}(x),$$ 
		and let $\mu_{2}=\sum_{y \in K_{2}}e^{-m_{2}(y)s}\delta_{y}$. Obviously, the elements in $\{\overline{B}_{FK_{m_{2}(y)}}(y,\epsilon)\}_{y \in K_{2}}$ are disjoint by \eqref{step1-disjoint} and (2-$b$). Hence, we can pick $0<\gamma_{2}<\frac{\gamma_{1}}{4}$ small enough, such that for any function $z:K_{2} \rightarrow X$ with $d(x,z(x))<\gamma_{2}$, we have for every $x \in K_{2}$,
		\begin{equation}
		(\overline{B}(z(x),\gamma_{2})) \cup (\overline{B}_{FK_{m_{2}(x)}}(z(x),\epsilon)) 
		\cap (\bigcup_{y \in K_{2} \setminus \{x\}}\overline{B}(z(y),\gamma_{2}) \cup \overline{B}_{FK_{m_{2}(y)}}(z(y),\epsilon))=\emptyset
		\end{equation}
		And we can find a large enough $n_{2} \in \mathbb{N}$  such that $K_{2} \subseteq Z_{n_{1}n_{2}}$, and 
		$$\mathcal{P}_{FK}(T,d,Z_{n_{1}n_{2}} \cap B(x,\frac{\gamma_{2}}{4}),t,\epsilon)>0$$ for each $x \in K_{2}$.
		
		\medskip
		
		{\noindent \bf{Step 3:}} Assume for $i=1,2,...,p$, $K_{i}$, $\mu_{i}$, $m_{i}(\cdot)$, $n_{i}$ and $\gamma_{i}$ have been constructed. Next we will construct
		$K_{p+1}$, $\mu_{p+1}$, $m_{p+1}(\cdot)$, $n_{p+1}$ and $\gamma_{p+1}$.
		
		We have the following conclusions. For any function $z:K_{p} \rightarrow X$ with $d(x,z(x))<\gamma_{p}$, we have for every $x \in K_{p}$,
		\begin{equation}\label{disjoint-p}
		(\overline{B}(z(x),\gamma_{p})) \cup (\overline{B}_{FK_{m_{p}(x)}}(z(x),\epsilon)) \cap \left(\bigcup_{y \in K_{p} \setminus \{x\}}\overline{B}(z(y),\gamma_{p}) \cup \overline{B}_{FK_{m_{p}(y)}}(z(y),\epsilon)\right)=\emptyset
		\end{equation}
		and $\mathcal{P}_{FK}(T,d,Z_{n_{1},n_{2},\dots,n_{p}} \cap B(x,\frac{\gamma_{p}}{4}),t,\epsilon)>0$ for each $x \in K_{p}$, $K_{p} \subseteq Z_{n_{1},n_{2},\dots,n_{p}}$.

		Note that the balls in $\{\overline{B}(x,\gamma_{p})\}_{x \in K_{p}}$ are pairwise disjoint. For each $x \in K_p$, since $\mathcal{P}_{FK}(T,d,Z_{n_{1},n_{2},\dots,n_{p}} \cap B(x,\frac{\gamma_{p}}{4}),t,\epsilon)>0$, similarly as in step 2, we can construct a finite set
		$$E_{p+1}(x) \subseteq Z_{n_{1},n_{2},\dots,n_{p}} \cap B \left(x,\frac{\gamma_{p}}{4} \right),$$
		and a integer-valued map
		$$m_{p+1}:E_{p+1}(x) \rightarrow \mathbb{N} \bigcap [\max\{m_{p}(y):y \in K_{p}\},\infty)$$
		such that
		
		(3-$a$)$\mathcal{P}_{FK}(Z_{n_{1},n_{2},\dots,n_{p}} \cap G, t, \epsilon)>0$, where $G$ is a open set with $G \cap E_{p+1}(x) \neq \emptyset$;
		
		(3-$b$)The elements in $\{\overline{B}_{FK_{m_{p+1}(y)}}(y,\epsilon)\}_{y \in E_{p+1}(x)}$ are disjoint, and
		$$\mu_{p}(\{x\})<\sum_{y \in E_{p+1}(x)}e^{-m_{p+1}(y)s}<(1+2^{-p-1})\mu_{p}(\{x\}).$$

		By \eqref{disjoint-p}, the balls in  $\{\overline{B}(x,\gamma_{p})\}_{x \in K_{p}}$ are disjoint, hence for any $x,x' \in K_{p}$ with $x \neq x'$, $E_{p+1}(x) \cap E_{p+1}(x')=\emptyset$.
		We denote $K_{p+1}=\bigcup_{x \in K_{p}}E_{p+1}(x),$
		and let $\mu_{p+1}=\sum_{y \in K_{p+1}}e^{-m_{p+1}(y)s}\delta_{y}$. 
		
		Note that the elements in $\{\overline{B}_{FK_{m_{p+1}(y)}}(y,\epsilon)\}_{y \in K_{p+1}}$ are disjoint by \eqref{disjoint-p} and  (3-$b$). Hence we can pick $0<\gamma_{p+1}<\frac{\gamma_{p}}{4}$ small enough, such that for any function $z:K_{p+1} \rightarrow X$ with $d(x,z(x))<\gamma_{p+1}$, we have for every $x \in K_{p+1}$,
		\begin{eqnarray*}
			(\overline{B}(z(x),\gamma_{p+1}) &\cup& (\overline{B}_{FK_{m_{p+1}(x)}}(z(x),\epsilon))\\
			&\cap& \left(\bigcup_{y \in K_{p+1} \setminus \{x\}}\overline{B}(z(y),\gamma_{p+1}) \cup \overline{B}_{FK_{m_{p+1}(y)}}(z(y),\epsilon)\right)=\emptyset
		\end{eqnarray*}
		For such $\gamma_{p+1}$, we can find a large enough $n_{p+1} \in \mathbb{N}$  such that $K_{p+1} \subseteq Z_{n_{1},n_{2},\dots,n_{p+1}}$, and   $\mathcal{P}_{FK} (Z_{n_{1},n_{2},\dots,n_{p+1}} \cap B(x,\frac{\gamma_{p+1}}{4}),t,\epsilon)>0$ for each $x \in K_{p+1}$.

		Now we have finished the construction.

	\end{proof}

	Using the lemma above, we can begin to prove Theorem \ref{thm-packing}.
	\begin{proof}

		Let $0<\epsilon<1$. We assume that $\mathcal{P}_{FK}(T,K,d,\epsilon)>0$. For any $0<s<\mathcal{P}_{FK}(T,K,d,\epsilon)$, by Lemma \ref{lem-packing-lower}, there is a compact set $K'\subset K$ and $\mu\in M(X)$ with $\mu(K')=1$ (hence  $\mu(K)=1$) such that  $\overline{h}_{\mu}^{FK}(T,\epsilon)\geq s$, and then we obtain that $$\overline{h}_{\mu}^{FK}(T,\epsilon)\geq\mathcal{P}_{FK}(T,K,d,\epsilon),$$ by letting $s\to\mathcal{P}_{FK}(T,K,d,\epsilon)$. Hence
		\begin{eqnarray*}
			\overline{mdim}_{FK}^P(T,K,d)&=&\limsup_{\epsilon\to0}\frac{\mathcal{P}_{FK}(T,K,d,\epsilon)}{\log \frac{1}{\epsilon}}\\
			&\le& \limsup_{\epsilon\to0}\frac{1}{\log\frac{1}{\epsilon}}\sup\{\overline{h}_{\mu}^{FK}(T,\epsilon):\mu(K)=1,\mu\in M(X)\}.
		\end{eqnarray*}
		
		Next we will prove the reverse inequility.
		Let $\epsilon>0$, $\mu\in M(X)$ with $\mu(K)=1$, we assume that $\overline{h}_{\mu}^{FK}(T,\epsilon)>0$. Let $0<s<\overline{h}_{\mu}^{FK}(T,\epsilon)$. We can choose a $\theta>0$, and a Borel set $A\subset K$ with $\mu(A)>0$ such that for all $x \in A$,
		$$\overline{h}_{\mu}^{FK}(T,x,\epsilon)=\limsup_{n\to\infty}-\frac{\log\mu(B_{FK_n}(x,\epsilon))}{n}>s+\theta.$$

		\noindent{\bf Claim:} For any $E\subset A$ with $\mu(E)>0$, we have $P_{FK}(T,d,E,s,\frac{\epsilon}{5})=\infty$.
		
		We first assume the claim holds. 
		Let $A_i \subset X, i=1,2,\dots,$ with $A \subset \bigcup_{i=1}^{\infty}A_i$. Then there exist some $A_i$ such that 
		$\mu(\overline{A}_i \cap A )>0$. Note that $\overline{A}_i \cap A \subset A$ is a Borel set, then by the claim and Lemma \ref{lem-2}, we have
		\begin{eqnarray*}
			\sum_{i}P_{FK}(T,A_i,d,s,\epsilon/10) 
			\ge  P_{FK}(T,\overline{A}_i,d,s,\epsilon/5) 
			\ge P_{FK}(T,\overline{A}_i\cap A,d,s,\epsilon/5)=\infty.
		\end{eqnarray*}
		
		Hence $$ \mathcal{P}_{FK}(T,A,d,s,\epsilon/10)
		=\inf \{\sum_{i=1}^{\infty} P_{FK}(T,A_i,d,s,\epsilon/10) :\bigcup_{i=1}^{\infty}A_{i}\supseteq A\}=\infty,$$
		and 
		$$\mathcal{P}_{FK}(T,K,d,s,\epsilon/10)\ge \mathcal{P}_{FK}(T,A,d,s,\epsilon/10)=\infty,$$
		which implies $\mathcal{P}_{FK}(T,d,K,\frac{\epsilon}{10})\ge s$.		
		Letting $s\to\overline{h}_{\mu}^{FK}(T,\epsilon)$,  we know $$\overline{h}_{\mu}^{FK}(T,\epsilon)\leq \mathcal{P}_{FK}(T,d,K,\frac{\epsilon}{10}).$$ Then$$\sup\{\overline{h}_{\mu}^{FK}(T,\epsilon):\mu\in M(X),\mu(K)=1\} \leq\mathcal{P}_{FK}(T,d,K,\frac{\epsilon}{10}),$$
		hence
		$$\limsup_{\epsilon\to0}\frac{1}{\log\frac{1}{\epsilon}}\sup\{\overline{h}_{\mu}^{FK}(T,\epsilon):\mu(K)=1,\mu\in M(X)\}\leq\overline{mdim}_{FK}^P(T,K,d).$$
		
		Next we prove the claim. Let $E\subset A$ with $\mu(E)>0$. For $n \in \N$, we define
		$$E_n=\{x\in E:\mu(B_{FK_n}(x,\epsilon))<e^{-(s+\theta)n}\}.$$	
		Since $E \subset A$,  we have $E=\cup_{n\geq N}E_n$ for any $N\in\mathbb{N}$. Fix  $N\in \N$, since $\mu(E)=\mu(\cup_{n\geq N}E_n)$, there is  $n\geq N$ such that 
		$$\mu(E_n)\geq\frac{1}{n(n+1)}\mu(E).$$
		
		Fix such $n$, consider a family of open covers $\{B_{FK_n}(x,\frac{\epsilon}{5})\}$ of $E_n$. By Lemma \ref{lem-5r}(using the Bowen metric $d_{FK_n}$ instead of $d$), there exists a finite pairwise disjoint subfamily $\{B_{FK_n}(x_i,\frac{\epsilon}{5})\}_{i\in I}$ , where $I$ is a finite index set, such that\[\cup_{i\in I}B_{FK_n}(x_i,\epsilon)\supset\cup_{x\in E_n}B_{FK_n}(x,\frac{\epsilon}{5})\supset E_n.\]
		
		Hence,
		\begin{equation*}
		\begin{split}
		P_{FK}(T,d,E,N,s,\frac{\epsilon}{5})
		&\geq P_{FK}(T,d,E_n,N,s,\frac{\epsilon}{5})\\
		&\geq\sum_{i\in I}e^{-sn}=e^{n\theta}\sum_{i\in I}e^{-n(s+\theta)}\\
		&\geq e^{n\theta}\sum_{i\in I}\mu(B_{FK_n}(x_i,\epsilon))\geq e^{n\theta}\mu(E_n)\\
		&\geq e^{n\theta}\frac{\mu(E)}{n(n+1)}.
		\end{split}
		\end{equation*}
		
		Letting $N\to\infty$, we obtain that $P_{FK}(T,d,E,s,\frac{\epsilon}{5})=\infty$. 
		
	\end{proof}

\end{document}